\documentclass[4pt, oneside]{article}

\usepackage{amsmath,amssymb,amsthm,units,stmaryrd}
\usepackage{qtree,bussproofs}

\usepackage{yfonts}
\usepackage{units,stmaryrd}
\usepackage{color}
\usepackage{graphicx}
\usepackage[all]{xy}
\newtheorem{theorem}{Theorem}[section]

\newtheorem{definition}[theorem]{Definition}
\newtheorem{lemma}[theorem]{Lemma}

\newtheorem{corollary}[theorem]{Corollary}
\newtheorem{proposition}[theorem]{Proposition}
\newtheorem{example}[theorem]{Example}
\newtheorem{notation}[theorem]{Notation}

\newcommand\Worms{{\mathbb W}}

\newcommand{\glp}{{\ensuremath{\mathsf{GLP}}}\xspace}
\newcommand{\pair}[1]{{\ensuremath{\langle #1 \rangle}}\xspace}

\usepackage{xspace}

\newcommand{\PRA}{\ensuremath{{\mathrm{PRA}}}\xspace}

\newcommand{\pa}{\ensuremath{{\mathrm{PA}}}\xspace}
\newcommand{\Robinson}{\ensuremath{{\mathrm{Q}}}\xspace}

\newcommand{\isig}[1]{{\ensuremath {\mathrm{I}\Sigma_{#1}}}\xspace}

\newcommand{\ea}{\ensuremath{{\rm{EA}}}\xspace}

\newcommand{\la}{\langle}
\newcommand{\ra}{\rangle}

\def\le{{\ell}}

\def\fmodels{\xymatrix{
\ar@{|=}[r]^{<\omega}&
}
}
\def\nmodels{\xymatrix{
\ar@{|=}[r]^{N}&
}
}
\def\<{\left <}

\def\>{\right >}

\DeclareSymbolFont{AMSb}{U}{msb}{m}{n}
\DeclareMathSymbol{\N}{\mathbin}{AMSb}{"4E}
\DeclareMathSymbol{\Z}{\mathbin}{AMSb}{"5A}
\DeclareMathSymbol{\R}{\mathbin}{AMSb}{"52}
\DeclareMathSymbol{\Q}{\mathbin}{AMSb}{"51}
\DeclareMathSymbol{\I}{\mathbin}{AMSb}{"49}
\DeclareMathSymbol{\C}{\mathbin}{AMSb}{"43}

\newcommand{\bt}{\begin{theorem}}
\newcommand{\et}{\end{theorem}}
\newcommand{\bl}{\begin{lemma}}
\newcommand{\el}{\end{lemma}}

\def\worms{{\mathbb W}}

\newcommand\remove[1]{}

\begin{document}

\title{Turing-Taylor expansions for arithmetic theories}
\author{Joost J. Joosten}

\maketitle

\begin{abstract}
Turing progressions have been often used to measure the proof-theoretic strength of mathematical theories: iterate adding consistency of some weak base theory until you ``hit" the target theory. Turing progressions based on $n$-provability give rise to a $\Pi_{n+1}$ proof-theoretic ordinal $|U|_{\Pi^0_{n+1}}$. As such, to each theory $U$ we can assign the sequence of corresponding $\Pi_{n+1}$ ordinals $\la |U|_n\ra_{n>0}$. We call this sequence a \emph{Turing-Taylor expansion} or \emph{spectrum} of a theory. 

In this paper, we relate Turing-Taylor expansions of sub-theories of Peano Arithmetic to Ignatiev's universal model for the closed fragment of the polymodal provability logic $\glp_\omega$. In particular, we observe that each point in the Ignatiev model can be seen as Turing-Taylor expansions of formal mathematical theories.

Moreover, each sub-theory of Peano Arithmetic that allows for a Turing-Taylor expansion will define a unique point in Ignatiev's model.
\end{abstract}

\section{Introduction}
Alan Turing considered in his dissertation progressions that are based on transfinitely adding consistency statements (\cite{Turing:1939:TuringProgressions}). If we disregard for the moment subtle coding and representation issues, these Turing progressions starting with some base theory $T$ were defined by
\[
\begin{array}{llll}
T^0 &:=& T;  \\
T^{\alpha +1} & :=& T^\alpha \cup \{ \, {\tt Con}({T^\alpha}) \, \}; & \\
T_\lambda & := & \bigcup_{\alpha < \lambda} T^\alpha & \mbox{for limit $\lambda$.} 
\end{array}
\]
Here, ${\tt Con}({T^\alpha})$ denotes some natural formalization of the statement that the theory ${T^\alpha}$ cannot derive, say, $0=1$. If one starts out with a sound base theory $T$ this gives rise to a progression of increasing proof-theoretic strength. Since the consistency statements are of logical complexity $\Pi^0_1$, Turing progressions can be used to define a $\Pi^0_1$ ordinal of a theory that contains (interprets) arithmetic; one starts out with a relatively weak theory $T$ and defines the $\Pi^0_1$ ordinal of some target theory $U$ by 
\[
|U|_{\Pi^0_1} \ := \ \sup \{ \alpha \mid T^\alpha \subseteq U \}.
\]
Using stronger notions of provability this can be generalized. We shall use $[n]_T$ to denote a formalization of ``provable in $T$ together with all true $\Pi^0_n$ sentences" and $\la n \ra_T$ will denote the dual consistency notion $\neg [n] \neg$. Generalized Turing progressions are readily defined:

\[
\begin{array}{llll}
T^0_n &:=& T;  \\
T^{\alpha +1}_n & :=& T^\alpha_n \cup \{ \la n \ra_{T^\alpha_n} \top\}; & \\
T_n^\lambda & := & \bigcup_{\alpha < \lambda} T_n^\alpha & \mbox{for limit $\lambda$.} 
\end{array}
\]
Here, the $\top$ stands for some fixed provable like for example $1=1$ so that $\la n \ra_{T^\alpha_n} \top$ simply says that the theory $T^\alpha_n$ is consistent with all true $\Pi_n$ formulas. We can now define the $\Pi^0_{n+1}$ proof-theoretical ordinal of a theory $U$ w.r.t.\  some base theory $T$:
\[
|U|_{\Pi^0_n} \ := \ \sup \{ \alpha \mid T_n^\alpha \subseteq U \}.
\]
Using Primitive Recursive Arithmetic as base theory, U. Schmerl proved in \cite{Schmerl:1978:FineStructure} that $|\pa|_{\Pi^0_n} = \varepsilon_0$ for all $n\in \omega$ and Beklemishev showed (\cite{Beklemishev:2003:ProofTheoreticAnalysisByIteratedReflection, Beklemishev:2004:ProvabilityAlgebrasAndOrdinals, Beklemishev:2005:Survey}) how provability logics can naturally be employed to perform and simplify the computations to obtain these ordinals.

In this paper we shall see how various theories can be written as the finite union of Turing progressions in a way reminiscent of how $\mathcal{\mathbf C}^\infty$ functions can be written as a countable sum of monomials in their Taylor expansion. Hence, we shall speak of \emph{Turing-Taylor} expansions of arithmetical theories. Whereas the monomials in a Turing expansion of a $\mathcal{\mathbf C}^\infty$ function are in a sense orthogonal, the monomials in our Turing-Taylor expansions are not. Therefore, we will sometimes call the Turing-Taylor expansions also \emph{ordinal spectra} or simply \emph{spectra} of theories.

\section{Arithmetical preliminaries}
We need to formalize various arguments that use cut-elimination. To this end, we assume that the base theory proves $\sf supexp$, i.e.~the totality of the super-exponential function $x\mapsto 2^x_x$, where $2^x_0 := x$ and $2^x_{y+1}:= 2^{2^x_y}$. However, we also need that our base theories are of low logical complexity aka, that the axioms are of logical complexity at most $\Pi^0_1$. 

To this end, we shall assume that any theory $T$ will be in a language that contains a function symbol for the super-exponentiation and that the recursive defining equations for this super-exponentiation are amongst the axioms of $T$. 

After having fixed our language, we define the arithmetical hierarchy syntactically as usual: $\Delta_0$ formulas are those formulas that only employ bounded quantification (i.e., quantification of the form $\forall \, x{<}t$ where $t$ is some term not containing $x$); If $\phi \in \Pi_n$ ($\Sigma_n$ resp.), then $\exists \, \vec x \ \phi \in \Sigma_{n+1}$ ($\forall \vec x \phi \in \Pi_{n+1}$ resp.).

Since $T$ has a constant for super-exponentiation, $T$ will be able to \emph{prove} the totality of super-exponentiation in a trivial way using induction for $\Delta_0$ formulas. It is folklore that $\Delta_0$ induction can be axiomatized in a $\Pi_1$ fashion:

\begin{lemma}
Over Robinson's arithmetic \Robinson the following two schemes are equivalent
\begin{enumerate}
\item\label{item:Pi2FormulationInduction} 
$\forall x\ (\forall\, y{<}x \, \phi(y) \to \phi(x)) \to \forall x \ \phi(x)$ for $\phi \in \Delta_0$;

\item\label{item:Pi1FormulationInduction}
$\forall x\ \Big(\forall\, z{\leq} x\ \big[\forall\, y{<}z\ \phi (y) \to \phi(z)\big]\to \phi (x)\Big)
$ for $\phi \in \Delta_0$. 
\end{enumerate}
\end{lemma}

\begin{proof}
The only non-trivial direction is $\eqref{item:Pi2FormulationInduction} \Rightarrow \eqref{item:Pi1FormulationInduction}$ which follows by applying \eqref{item:Pi2FormulationInduction} to $\phi'(x,u) \ :=\ x\leq u \to \phi(x)$.
\end{proof}

In the paper we shall heavily use formalized provability and the corresponding provability logics. As such, for c.e.\ theories $T$ we fix natural formalizations $[n]_T$ of ``provable in $T$ together with all true $\Pi_n$ sentences'' of complexity $\Sigma_{n+1}$ and the dual consistency notion $\la n\ra_T$ of complexity $\Pi_{n+1}$. When the context allows us to, we shall drop mention of the base theory $T$ and moreover, instead of writing $[0]$ ($\la 0 \ra$) we often write $\Box$ ($\Diamond$). 

We shall typically refrain from distinguishing a formula $\phi$ from its G\"odel number or even a natural syntactical term denoting its G\"odel number. Also, we use the standard dot notation $\Box \, \phi(\dot x)$ to denote a formula with free variable $x$ so that for each $x$ the formula $\Box \, \phi (\dot x)$ is provably equivalent to $\Box\, n$ where $n$ is the G\"odel number of $\phi(t)$ where $t$ is some term (often called numeral) denoting $x$. Note that for non-standard $x$, the corresponding term denoting $x$ will also be non-standard. 

We shall assume that each c.e.~theory $T$ that we consider comes with a $\Delta_0$ formula that defines the set of G\"odel numbers of axioms of $T$ on the standard model.
A main result about formalized provability is formulated in what is nowadays called L\"ob's rule (\cite{Lob:1955:SolutionProblemHenkin}):

\begin{proposition}
Let $T$ be a theory extending \ea. If $T\vdash \Box \phi \to \phi$, then $T\vdash \phi$.
\end{proposition}

The natural way to prove statements about Turing progression is by transfinite induction. Weaker theories however cannot prove transfinite induction. Schmerl (\cite{Schmerl:1978:FineStructure}) introduced a way to circumvent transfinite induction employing so-called \emph{reflexive transfinite induction}.

\begin{lemma}[Reflexive transfinite induction]
Let $T$ be some theory extending say, \ea, so that 
\[
T \vdash \forall \alpha \ \Big(\Box_T \ \forall \, \beta{<}\dot \alpha \ \phi(\beta) \ \to \ \phi(\alpha)\Big).
\]
Then it holds that $T\vdash \forall \alpha \  \phi (\alpha)$.
\end{lemma}

\begin{proof}
Clearly, if $T \vdash \forall \alpha \Big(\Box_T \ \forall \, \beta{<}\dot \alpha \ \phi(\beta) \ \to \ \phi(\alpha)\Big)$, then also 
\[
T \vdash \Box_T \ \forall \alpha \ \phi(\alpha) \ \to \ \forall \alpha \ \phi(\alpha),
\]
and the result  follows from L\"ob's rule.
\end{proof}

For theories $U$ and $V$, we shall write $U\equiv_n V$ for the statement that $U$ and $V$ prove the same $\Pi_{n+1}$ formulas.

\section{Modal preliminaries}

We shall see that the polymodal provability logic $\glp_\omega$ is particularly well-suited to speak about Turing progressions and finite unions thereof.  

\subsection{Provability logics and worms}
We first define a polymodal version of provability logic as introduced by Japaridze in \cite{Japaridze:1988}.

\begin{definition}
The propositional polymodal provability logic $\glp_\omega$ has for each $n<\omega$ a modality $[n]$ with dual modality $\la n\ra$ being short for $\neg [n] \neg$. The language contains the constants $\top$ and $\bot$ for logical truth and falsity respectively. 

The rules of $\glp_\omega$ are Modus Ponens and Necessitation for each $[n]$ modality: $\frac{\phi}{[n]\phi}$. The axioms are
\begin{enumerate}
\item All propositional tautologies in the language of $\glp_\omega$;

\item
$[n](\phi \to \psi) \to ([n]\phi \to [n]\psi)$ for each $n<\omega$ and $\glp_\omega$ formulas $\phi$ and $\psi$;

\item
$[n]([n]\phi \to \phi) \to [n]\phi$ for each $n<\omega$ and $\glp_\omega$ formula $\phi$;

\item
$[n]\phi \to [m]\phi$ for each $n<m<\omega$ and each $\glp_\omega$ formula $\phi$;

\item
$\la n\ra\phi \to [m]\la n\ra \phi$ for each $n<m<\omega$ and each $\glp_\omega$ formula $\phi$;

\end{enumerate}
\end{definition}

\noindent
It is well-known that $[n]\phi \to [n][n]\phi$ is derivable in $\glp_\omega$ and we shall use that without specific mention. The logic $\glp_\omega$ is sound and complete for a wide range of theories $T$ when interpreting the modal operator $[n]$ as the formalized provability predicate $[n]_T$ (\cite{Japaridze:1988, Ignatiev:1993:StrongProvabilityPredicates}). 

A standing assumption throughout all this paper is that all theories that we consider yield soundness of $\glp_\omega$. Moreover, we shall assume that any theory $T$ contains $\ea^+$ and has a set of axioms whose set of G\"odel numbers is definable on the standard model by a $\Delta_0$ formula.

The closed fragment $\glp_\omega^0$ of $\glp_\omega$ consists of all those $\glp_\omega$ theorems that do not contain propositional variables. We define \emph{worms} to be the collection of iterated consistency statements within $\glp_\omega^0$ and denote them by $\Worms$:

\begin{definition}
For each $n<\omega$, the empty worm $\top$ is in $\Worms_n$; We inductively define that if $A\in \Worms_n$ and $\omega> m\geq n$, then $\la m\ra A \in \Worms_n$. The set $\Worms$ of all $\glp_\omega$ worms is just $\Worms_0$.
\end{definition}

\noindent
Often we shall just identify a worm with the string of subsequent modality indices denoting the empty string by $\top$ for convenience. We now define a convenient decomposition of worms that will allow for inductive proofs.

\begin{definition}
For a $\glp_\omega$ worm $A$, its \emph{n-head} --we write $h_n(A)$--is the left-most part of $A$ that consists of only modalities which are at least $n$. The remaining part of $A$ is called the \emph{$n$-remainder} and is denoted by $r_n(A)$. 

More formally: $h_n(\top) = \top$; and $h_n(mA) = mh_n(A)$ in case $m\geq n$ and $\top$ otherwise. Likewise: $r_n(\top)= \top$ and $r_n(mA) = r_n(A)$ in case $m\geq n$ and $mA$ otherwise. 
\end{definition}

\noindent
The following lemma whose proof we leave as an exercise turns out to be very useful.

\begin{lemma}
For each $\glp_\omega$-worm $A$ and for each $n<\omega$, we have 
\[
\glp_\omega\vdash A \leftrightarrow h_n(A) \ \wedge \ r_n(A).
\]
\end{lemma}

\subsection{Lost in translation}

It is well-known (\cite{Beklemishev:2005:Survey, BeklemishevFernandezJoosten:2014:LinearlyOrderedGLP}) that worms constitute an alternative ordinal notation systems if we order them by 
\[
A <_n B \ :\Leftrightarrow \ \glp_\omega \vdash B \to \la n\ra A.
\]
\begin{proposition}
$\la \varepsilon_0, < \ra \cong \la \Worms_n/\equiv, <_n\ra$. 
\end{proposition}
Here, $\Worms_n/\equiv$ denotes $\Worms_n$ modulo $\glp_\omega$ provable equivalence and $\la \varepsilon_0, <\ra$ is just the ordinal $\varepsilon_0 = \sup \{ \omega, \omega^{\omega}, \omega^{\omega^{\omega}}, \ldots\}$ under the usual ordinal ordering $<$. 

Worms can be related smoothly to more standard ordinal notations using the so-called hyper-exponentiation functions $e^n : {\sf On} \to {\sf On}$ (see \cite{FernandezJoosten:2012:Hyperations}) where $\sf On$ denotes the class of ordinals and $e^0$ is the identity function; $e^1 : \xi \mapsto -1 + \omega^\xi$; and $e^{n+m} = e^n\circ e^m$. The following theorem is proven in \cite{FernandezJoosten:2014:WellOrders}:

\begin{proposition}\label{theorem:wormCalculusForGLPworms}
Let $o_0: \Worms \to {\sf On}$ be defined by 
\begin{enumerate}
\item
$o_0(\top) = 0$;
\item
$o_0(B0A) = o_0(A) + 1 + o_0(B)$;
\item
$o_0(n\uparrow A) = e^n(o_0(A))$.
\end{enumerate}
Here, $n\uparrow A$ denotes the worm that arises by simultaneously substituting any modality $m$ in $A$ by $n+m$.

 Further, for any worm $A\in \Worms$ we define $o_n(n \uparrow A) := o_0(A)$ and $o_n(A) := o_n(h_n(A))$. We now have that 
\[
o_n:  \la \Worms_n/\equiv, <_n\ra  \cong \la \varepsilon_0, < \ra.
\]
\end{proposition}

In previous papers on polymodal provability logics, various proofs are rather involved since they work with classical ordinal notation systems. Worms have further logical and algebraic structure so that in the context of provability logics and Turing progressions, they are the better ordinal notation systems. 

\begin{notation}\label{notation:TuringProgressionsIndexedByWorms}
For $A{\in} \Worms$, by $T_n^A$ we shall denote the Turing progression $T_n^{o_n(A)}$.
\end{notation}

Note that in virtue of our definitions we have that $T^A_n = T^{h_n(A)}_n$ and we shall use both notations interchangeably. Moreover, we note that a worm can denote various objects: an iterated consistency statement in modal logic, an iterated consistency statement in the language of arithmetic, and an ordinal. The context will always reveal what kind of object the occurrence of a particular worm denotes and thus we refrain from separating the different possible denotations by introducing extra notation.

\section{Turing Taylor expansions}

We shall see that various theories can be written as the finite union of simple Turing progressions which we call the \emph{Turing-Taylor expansion}. We start by looking at theories axiomatized by worms. Recall that all our theories are in the language containing a symbol for super-exponentiation and are supposed to come with a $\Delta_0$ axiomatization.

\subsection{Worms and Turing progressions}

The generalized Turing progressions $T^n_A$ are not too sensitive to adding ``small" elements to the base theory as is expressed by the following lemma. 

\begin{lemma}\label{theorem:GeneralizedTuringProgressionsInvariantToSmallChangesBaseTheory}
For any theory $T$ and for any $\sigma \in \Sigma_{n+1}$, we have provably in $\ea^+$ that
\[
(T+\sigma)^\alpha_n \equiv (T)^\alpha_n +\sigma \ \ \ \ \ \mbox{for any $\alpha < \epsilon_0$}.
\]
In particular, for any theory $T$ and for any $GLP_\omega$ worm $A$, if $m<n<\omega$, then 
\[
(T+mA)^\alpha_n \equiv (T)^\alpha_n +mA \ \ \ \ \ \mbox{for any $\alpha < \epsilon_0$}.
\]
\end{lemma}

\begin{proof}
By a straight-forward reflexive transfinite induction using provable $\Sigma_{n+1}$-completeness at the inductive step: for $n<\omega$ and $\sigma \in \Sigma_{n+1}$, we have
\[
\sigma \to [n]_T \, \sigma.
\]
\end{proof}

\noindent
The main motor to relate provability logics to Turing progression is by means of the following theorem.

\begin{theorem}\label{theorem:wormsAndGeneralizedTuringProgressions}
Let $T$ be some elementary presented theory containing $\ea^+$ whose axioms have logical complexity at most $\Pi_{n+1}$ and let $A$ be some worm in $\Worms_n$. We have, provably in $\ea^+$, that 
\[
T+ A \equiv_n T_n^{A}.
\]
\end{theorem}

\begin{proof}
By reflexive transfinite induction. We refer to \cite[Theorem 17]{Beklemishev:2005:Survey} for details.
\end{proof}

In general we do of course not have\footnote{It is known that $\PRA + \neg \isig{1}$ and \isig{1} are $\Pi_2^0$ equivalent (see \cite[Lemma 3.4]{Joosten:2005:ClosedFragmentILPRAwithIsig1}). Clearly, $\PRA + \neg \isig{1} + \isig{1} \not \equiv_1 \isig{1}$.} that if $U\equiv_{n}V$, then $U + \psi \equiv_n V + \psi$ for theories $U$ and $V$ and formulas $\psi$. However, in the case of Turing progressions we can include ``small" additions on both sides and preserve conservativity.

\begin{lemma}\label{theorem:addingSmallWormsToT+A}
Let $T$ some theory whose axioms have logical complexity at most $\Pi_{n+1}$ and let $A$ be some worm in $\Worms_n$. Moreover, let $B$ be any worm and $m<n$. We have, verifiably in $T$, that 
\[
T + A + mB  \equiv_n T_n^{A} + mB.
\] 
\end{lemma}

\begin{proof}
As $m<n$ we have that $mB \in \Pi_{n}$. Whence, we can apply Theorem \ref{theorem:wormsAndGeneralizedTuringProgressions} to the theory $T+mB$ and obtain
\[
T + mB + A \equiv_n (T+ mB)_n^{A}
\] 
However, by Lemma \ref{theorem:GeneralizedTuringProgressionsInvariantToSmallChangesBaseTheory} we see that 
\[
(T+ mB)_n^{A}\equiv T_n^{A} + mB, \ \mbox{ whence } \ T + mB + A \equiv_n T_n^{A} + mB.
\]
\end{proof}
\noindent
From this lemma we obtain the following simple but very useful corollary.

\begin{corollary}\label{theorem:T+AIsAlmostATuringProgression}
Let $T$ be some theory whose axioms have logical complexity at most $\Pi_{n+1}$. Moreover, let $A$ be any worm. We have verifiably in $T$ that 
\[
T + A \ \equiv_n \  T_n^{h_n(A)} + r_n(A) \ \equiv_n \ T_n^{A} + r_n(A).
\]
\end{corollary}

\begin{proof}
Since $\glp \vdash A \leftrightarrow h_n(A) \wedge r_n(A)$ and since by assumption $\glp_\omega$ is sound w.r.t.\ $T$ we see that $T+ A \equiv T+ h_n(A) + r_n(A)$.
The worm $r_n(A)$ is either empty or of the form $mA$ for some $m<n$. Clearly, $h_n(A) \in \Worms_n$. Thus, we can apply Lemma \ref{theorem:addingSmallWormsToT+A} and obtain
\[
T+ h_n(A) + r_n(A) \equiv_n T_n^{h_n(A)} + r_n(A).
\]
Recall that by our notation convention (Notation \ref{notation:TuringProgressionsIndexedByWorms}) and by the definition of $o_n$ (see Proposition \ref{theorem:wormCalculusForGLPworms}), we have that $T_n^{h_n(A)} + r_n(A)\equiv T_n^{A} + r_n(A)$.
\end{proof}

\subsection{Theories axiomatized by worms}
From Theorem \ref{theorem:wormsAndGeneralizedTuringProgressions} we see that we can capture the $\Pi_1^0$ consequences of the $o(A)$-th Turing Progression of $T$ by the simply axiomatized theory $T+A$. 

That is, $T+A$ proves the same $\Pi^0_1$ formulas as $T^0_{A}$. However, $T+A$ will in general prove many new formulas of higher complexity. We can characterize those consequences of $T+A$ also in terms of Turing progressions as we see in the next theorem.

\begin{theorem}\label{theorem:OmegaSequenceInTuringProgressions}
Let $T$ be some $\Pi_1^0$ axiomatizable theory. Let $A$ be any $\glp_\omega$ worm. We have, verifiably in $T$, that 
\begin{enumerate}
\item
$T + A \equiv \bigcup_{i< \omega} T_i^A$, and

\item
$T + A \equiv_n \bigcup_{i=0}^{n} T_i^A$.
\end{enumerate}
\end{theorem}

\begin{proof}
It suffices to prove the second item in the form of $T + A \equiv_n \bigcup_{i=0}^{n} T_i^{h_i(A)}$ since for any worm $A$, we have $h_i(A) =\top$ for $i$ large enough. We prove the second item by an external induction on $n$ and the base case follows directly from Theorem \ref{theorem:wormsAndGeneralizedTuringProgressions}.

For the inductive case we reason in $T$ as follows. By Corollary \ref{theorem:T+AIsAlmostATuringProgression} we know that
\begin{equation}\label{IAmTiredOfTheseLengthyLables}
T + A \equiv_{n+1} T_{n+1}^{h_{n+1}(A)} + r_{n+1}(A).
\end{equation}
In particular, as $T_{n+1}^{h_{n+1}(A)} + r_{n+1}(A) \subseteq \Pi_{n+2}$ we see that actually, $T+A$ is a $\Pi_{n+2}$-conservative extension of $T_{n+1}^{h_{n+1}(A)} + r_{n+1}(A)$, that is,
\[
T + A \vdash T_{n+1}^{h_{n+1}(A)} + r_{n+1}(A).
\]
The induction hypothesis tells us that 
\begin{equation}\label{shortLabel}
T + A \equiv_n \bigcup_{i=0}^{n} T_i^{h_i(A)}.
\end{equation}
Again, since $\bigcup_{i=0}^{n} T_i^{h_i(A)} \subseteq \Pi_{n+1}$ we obtain that 
\[
T + A \vdash \bigcup_{i=0}^{n} T_i^{h_i(A)}.
\]
Thus, $T+A \vdash \bigcup_{i=0}^{n+1} T_i^{h_i(A)}$ and in particular, if $\bigcup_{i=0}^{n+1} T_i^{h_i(A)} \vdash \pi$ then $T+A \vdash \pi$ for $\pi \in \Pi_{n+2}$. 

Conversely, assume that $T+A \vdash \pi$ for some $\Pi_{n+2}$ sentence $\pi$. By \eqref{IAmTiredOfTheseLengthyLables} we see that $T_{n+1}^{h_{n+1}(A)} + r_{n+1}(A) \vdash \pi$. However, $r_{n+1}(A) \in \Pi_{n+1}$ and $T+A \vdash r_{n+1}(A)$ so, by \eqref{shortLabel} we see that $\bigcup_{i=0}^{n} T_i^{h_i(A)} \vdash r_{n+1}(A)$. Thus 
\[
\begin{array}{lll}
\bigcup_{i=0}^{n+1} T_i^{h_i(A)} & \vdash & T_{n+1}^{h_{n+1}(A)} + r_{n+1}(A) \\
\ & \vdash & \pi.
\end{array}
\]
as was required.

\end{proof}

It is clear that the modal reasoning in the proof of Theorem \ref{theorem:OmegaSequenceInTuringProgressions} can be extended beyond $\glp_\omega$. For this, one needs a (hyper-)arithmetic interpretation of $\glp_\Lambda$. One such example is given in \cite{FernandezJoosten:2013:OmegaRuleInterpretationGLP} but for $\lambda > \omega$ there are no canonical formula complexity classes around. This problem can be solved by considering a different interpretation of $\glp_\Lambda$ as presented in \cite{Joosten:2015:JumpsThroughProvability}.

As a nice corollary to Theorem \ref{theorem:OmegaSequenceInTuringProgressions} we get the following simple but useful lemma.

\begin{lemma}\label{theorem:HigherProgressionImpliesLowerOnesWormVersion}
Let $T$ be a $\Pi_{m+1}$ axiomatized theory. For $m\leq n$ and $A\in \Worms_n$, we have, verifiably in $T$, that $T_n^A  \vdash T_{m}^A$.
\end{lemma}

The restriction on the complexity of $T$ can actually be dropped as was shown in \cite{Schmerl:1978:FineStructure, Beklemishev:1995:IteratedLocalReflectionVersusIteratedConsistency}.

\section{Ignatiev's model and Turing-Taylor expansions}

In this section we shall focus on sub-theories of Peano Arithmetic. We shall see that if such a theory can be written as the finite union of generalized Turing progressions, then it can be seen as ``an element" of a well-known model for modal logic.

\subsection{Ignatiev's universal model}

The closed fragment $\glp^0_\omega$ of $\glp_\omega$ is a rich yet decidable structure. Ignatiev exposed a Kripke model for $\glp^0_\omega$ that is universal in that the set of all formulas valid in all worlds of the model is exactly the set of theorems of $\glp^0_\omega$.

We refer to the standard literature (\cite{Ignatiev:1993:StrongProvabilityPredicates, Joosten:2004:InterpretabilityFormalized, BeklemishevJoostenVervoort:2005:FinitaryTreatmentGLP}) for details and limit ourselves here to defining the model and state its main properties.

Ignatiev's universal model $\mathcal U$ is a pair $\la \mathcal{I}_\omega, \{  \succ_i \}_{i\in \omega} \ra$ where $\mathcal{I}_\omega$ is a set of \emph{worlds} and for each $i\in \omega$ we have a binary relation $\succ_i$ on $\mathcal{I}_\omega$. Worlds in Ignatiev's model for $\glp_\omega$ are sequences of ordinals, 
\[
\la \alpha_0, \alpha_1, \alpha_2, \ldots \ra \mbox{ with $\alpha_{n+1}\leq\le(\alpha_n)$}
\]
where $\le (\alpha + \omega^\beta) = \beta$ and $\le(0)=0$. We define $\vec \alpha \succ_n \vec \beta$ to hold exactly when both $\alpha_i = \beta_i$ for all $i<n$ and $\alpha_n > \beta_n$. We have included a picture of a part of $\mathcal U$ in Figure \ref{fig:ignatiev}.

We recursively define $\vec \alpha \nVdash \bot$, $\vec \alpha \Vdash \phi \to \psi$ iff ($\vec \alpha \nVdash \phi$ or $\vec \alpha \Vdash \psi$), and $\vec \alpha \Vdash [n] \phi$ iff for all $\vec \beta$ with $\alpha \succ_n \beta$ we have $\vec \beta \Vdash \phi$, with the tacit understanding that the other connectives are defined as usual in terms of $\bot, \to$ and the $[n]$'s. Ignatiev's model $\mathcal U$ is universal in that $\glp^0_\omega \vdash \phi \ \ \Leftrightarrow \ \ \forall \, \vec \alpha {\in} \mathcal{I}_\omega\ \ \vec \alpha \Vdash \phi$.

In the light of Proposition \ref{theorem:wormCalculusForGLPworms}, we can represent each world $\vec \alpha \in \mathcal{I}_\omega$ by a (non-unique) sequence of worms $A_n\in \Worms_n$:
\[
\la A_0, A_1, A_2, \ldots \ra \mbox{ with $A_{n+1}\leq_{n+1}h(A_n)$ and $o_n(A_n) = \alpha_n$.}
\]
We often refer to these sequences $\vec A$ as \emph{Ignatiev sequences} with their condition on them that $A_{n+1}\leq_{n+1}h(A_n)$.

\subsection{Turing-Taylor expansions}

Since $\ea^+$ seems to be weakest theory for which we can properly formalize Turing progressions we will define the $\Pi^0_{n+1}$ ordinal a theory $U$ to be 
\[
|U|_{\Pi^0_n} \ := \ \sup \{ \alpha \mid (\ea^+)_n^\alpha \subseteq U \}.
\]
A main point of this paper is that collecting these ordinals per theory yields in interesting structure: Ignatiev's model $\mathcal U$. Therefore, we define the spectrum of a theory as expected.

\begin{definition}
For $U$ a formal arithmetic theory that lies in between $\ea^+$ and \pa, we define its \emph{spectrum} $tt(U)$ by
\[
tt(U)\ \ := \ \ \la |U|_{\Pi^0_1}, |U|_{\Pi^0_2}, |U|_{\Pi^0_3}, \ldots \ra.
\]
Suggestively, we shall also speak of the \emph{Turing-Taylor expansion} of $U$ instead of the spectrum of $U$. 
In case $U \equiv \bigcup_{n=0}^\infty (\ea^+)_n^{|U|_{\Pi^0_{n+1}}}$ we say that $U$ has a convergent Turing-Taylor expansion.
\end{definition}

We include the reference to Taylor in the name due to the analogy to Taylor expansions of $C^{\infty}$ functions, that is, functions that are infinitely many times differentiable (see acknowledgements). If $f$ is a $C^{\infty}$ function, one can consider its Taylor expansion around 0 as 
$f(x) = \sum_{n=0}^{\infty}a_n x^n$. Thus, each Taylor expansion is determined by by its sequence $\la a_0, a_1, a_2, \ldots \ra$ of coefficients. In the case of a convergent Turing-Taylor expansion we fully determine the expansion by a sequence of ordinals $\la \xi_0, \xi_1, \xi_2,\dots\ra$ so that 
\[
U \equiv \bigcup_{n=0}^\infty T_n^{\xi_n}.
\]
We shall study which sequences of ordinals are attainable as coming from a convergent Turing-Taylor expansion. 

Note that we have defined $tt(U)$ as to include only $\Pi_n^0$ sentences but this can easily be generalized to suitable sentences of higher complexities. For our current purpose, studying sub-theories of \pa, the restriction is not essential.

For Taylor expansions there is actually a uniform way of computing the coefficients as 
$f(x) = \sum_{n=0}^{\infty}\frac{f^{(n)}(0)}{n !}\  x^n$ where $f^{(n)}$ denotes the $n$-th derivative of $f$ and $f^{(0)} : = f$. For theories axiomatized by worms there we saw in Theorem \ref{theorem:OmegaSequenceInTuringProgressions} that there is also such a uniform way of computing the coefficients.

Note that the analogy to Taylor expansions is by no means perfect. In particular, in Taylor expansions we see that all the monomials $x^n$ are mutually independent, whereas in Turing progressions there will be certain dependency as we already saw in Lemma \ref{theorem:HigherProgressionImpliesLowerOnesWormVersion}. Therefore, we will rather speak of the spectrum or \emph{ordinal spectrum} of a theory instead of its Turing-Taylor expansion.
\bigskip

With every sequence $\vec \alpha = \la \alpha_0, \alpha_1, \ldots \ra$ of ordinals below $\varepsilon_0$ we can naturally associate a sub theory $(\vec \alpha)_{\sf tt}$ of \pa as follows
\[
(\vec \alpha)_{\sf tt} := \bigcup_{n=0}^{\infty} \ea_n^{\alpha_n}.
\]
Of course we can and shall write the $\alpha_n$ most of the times as worms $A_n$ in $\Worms_n$. In general, we do not have that $tt((\vec A)_{\sf tt}) = \vec A$. Let us first see this in a concrete example and then prove some general theorems in the next sections.

\begin{example}\label{example:TuringTaylerProjections}
For $\Pi^0_1$ axiomatized theories $T$ we have (in worm-notation) that $T^1_{1} + T_0^{01} \equiv T^1_1 + T_0^{101}$. In the classical notation system this reads $T_1^{1} + T_0^{\omega +1} \equiv T^1_1 + T_0^{\omega\cdot 2}$.
\end{example}

\begin{proof}
By Theorem \ref{theorem:OmegaSequenceInTuringProgressions} we have that $T^1_{1}\equiv T+\la 1 \ra \top$ and $T_0^{01} \equiv T+\la 0 \ra  \la 1 \ra \top$. Thus, $T^1_1 + T_0^{01} \equiv T + \la 1 \ra \top + \la 0 \ra \la 1\ra \top$. Clearly the latter is equivalent to $T + \la 1 \ra \la 0 \ra \la 1\ra \top$ and we obtain our result by one more application of Theorem \ref{theorem:OmegaSequenceInTuringProgressions}. 

Using Proposition \ref{theorem:wormCalculusForGLPworms} one gets the correspondence to the more familiar ordinal notation system.
\end{proof}

\subsection{Each Turing-Taylor expansion corresponds to a unique point in Ignatiev's model}

We shall now prove that for each theory $U$ we have that $tt(U)$ is a sequence that occurs in $\mathcal{I}_\omega$. Most of the work in doing so is included in the following theorem.

\begin{theorem}\label{theorem:EachTheoryIsApointInIgnatievsModel}
Let $T$ be some $\Pi_{n+1}$ axiomatized theory and let $A\in \Worms_{n+1}$ and $B\in \Worms_n$. We have, verifiably in $T$, that 
\[
T_{n+1}^A + T_n^{nB} \equiv_{n+1} T + A + nB, 
\]
and
\[
T_{n+1}^A + T_n^{nB} \equiv_n T_n^{AnB}. 
\]
\end{theorem}

\begin{proof}
Since $B\in \Worms_n$, by Theorem \ref{theorem:OmegaSequenceInTuringProgressions} and Lemma \ref{theorem:HigherProgressionImpliesLowerOnesWormVersion} we know that 
\[
T_n^{nB} \equiv T+nB.
\]
Consequently, we obtain the following equivalence.
\begin{equation}\label{equation:AddingFinitelyAxiomatizedTuringProgression}
T_{n+1}^A + T_n^{nB} \equiv T_{n+1}^A + nB
\end{equation}
Let us now see the following conservation result which proves the first part of the theorem.
\begin{equation}\label{equation:ConservationBetweenTwoTPsAndTwoWorms}
T_{n+1}^A + T_n^{nB} \equiv_{n+1} T+A +nB
\end{equation}
By \eqref{equation:AddingFinitelyAxiomatizedTuringProgression} it suffices to show that 
\[
T_{n+1}^A + nB \vdash \pi \ \ \Leftrightarrow \ \ T+A + nB \vdash \pi
\]
for any $\pi \in \Pi^0_{n+2}$. However, if $\pi \in \Pi^0_{n+2}$ we also have that $(nB \to \pi) \in \Pi^0_{n+2}$ since $nB \in \Pi^0_{n+1}$. Thus we can reason 
\[
\begin{array}{llll}
T_{n+1}^A + nB \vdash \pi & \Leftrightarrow & T_{n+1}^A \vdash  nB\to \pi& \\
& \Leftrightarrow & T+A \vdash  nB\to \pi& \mbox{by Theorem \ref{theorem:wormsAndGeneralizedTuringProgressions}}\\
& \Leftrightarrow & T+A + nB \vdash   \pi& \\
& \Leftrightarrow & T+AnB \vdash   \pi& \mbox{since $A\in \Worms_{n+1}$.}\\
\end{array}
\]
This proves \eqref{equation:ConservationBetweenTwoTPsAndTwoWorms} and also $T_{n+1}^A + T_n^{nB} \equiv_{n+1} T+AnB$. We readily obtain the second claim of our theorem since by Theorem \ref{theorem:wormsAndGeneralizedTuringProgressions} we have $T+AnB\equiv_n T_n^{AnB}$.
\end{proof}

\begin{corollary}\label{theorem:EachTtDefinesAUniquePoint}
If $U$ is some sub-theory of \pa with a convergent Turing-Taylor expansion, so that $U \not \equiv_0 \pa$, then $tt(U)$ defines a point in $\mathcal{I}_\omega$. 
\end{corollary}

\begin{proof}
Since $U$ has a convergent Turing-Taylor expansion, $|U|_{\Pi^0_1}$ is well-defined. Since it is well-known that $(\ea^+)^0_{\varepsilon_0} \equiv_0 \pa$, by the assumption that $U \not \equiv_0 \pa$ we know that we can find a $\glp_\omega$ worm $A$ with $(\ea^+)_0^A \equiv_0 U$ whence $|U|_{\Pi^0_1} = A$. 

By Lemma \ref{theorem:HigherProgressionImpliesLowerOnesWormVersion} we see that each $|U|_{\Pi^0_n}<\varepsilon_0$. Thus, indeed, $tt(U)$ defines a sequence $\vec A$ of worms. W.l.o.g.~we pick $\vec A$ so that $A_n \in \worms_n$ for each $n$.

Suppose now for a contradiction that for some $n$, the sequence $\vec A$ does not satisfy the condition that $A_{n+1}\leq_{n+1} h_{n+1}(A_n)$. As provably $A_n\leftrightarrow h_{n+1}(A_n)r_{n+1}(A_n)$, clearly, $A_n\geq_n r_{n+1}(A_n)$ whence $T_n^{A_n} \vdash T_n^{r_{n+1}(A_n)}$. By Theorem \ref{theorem:EachTheoryIsApointInIgnatievsModel} we know that $T_{n+1}^{A_{n+1}} + T_n^{r_{n+1}(A_n)}\vdash T_n^{A_{n+1}r_{n+1}(A_n)}$. But, since by assumption $A_{n+1}>_{n+1} h_{n+1}(A_n)$ we know that $A_{n+1}r_{n+1}(A_n) >_n A_n$. The latter violates the assumption that $|U|_{\Pi_n^0} = A_n$ is the supremum of all $B$ so that $(\ea^+)_n^B \subseteq U$.
\end{proof}

\subsection{Each point in Ignatiev's model corresponds to a unique Turing-Taylor expansion}

Corollary \ref{theorem:EachTtDefinesAUniquePoint} tells us that certain points in the Ignatiev model $\mathcal{I}_\omega$ can be seen as mathematical theories with a convergent Turing-Taylor expansion.
We now wish to see that \emph{every} point $\vec A$ in the Ignatiev model $\mathcal{I}_\omega$ can be interpreted naturally as a theory. 

The natural candidate would of course be the theory $(\vec A)_{\sf tt}$. But, we have already seen in Example \ref{example:TuringTaylerProjections} that in general we do not have
$tt((\vec A)_{\sf tt})=\vec A$. However, as we shall see, for points in the Ignatiev model the equality does hold. 

We will need two technical lemmas to deal with the adjacent points in $\vec A$. The first, Lemma \ref{theorem:IgnatievConditionViolated} deals with the case that these adjacent points violate the condition $A_{n+1}\leq_{n+1}h(A_n)$ of Ignatiev sequences. The second case, Lemma \ref{theorem:ConservationIgnatievStep} deals with the case when no such violation is there.

\begin{lemma}\label{theorem:IgnatievConditionViolated}
Let $T$ be a $\Pi_{n+1}$ axiomatized theory. Moreover, let $A\in \Worms_{n+1}$, $B\in \Worms_n$ and suppose $A\geq_{n+1} h_{n+1}(B)$. Then, verifiably in $T$, we have
\[
T_{n+1}^A + T_n^B \equiv_n T_n^{Ar_{n+1}(B)}.
\]
\end{lemma}
\begin{proof}
We reason in $T$. Since clearly $B\geq_{n}r_{n+1}(B)$ we have that $T_n^{B} \vdash T_n^{r_{n+1}(B)}$. Consequently, 
\[
\begin{array}{llll}
T_{n+1}^A + T_n^B & \vdash & T_{n+1}^A + T_n^{r_{n+1}B} & \\
 & \equiv_n & T_n^{Ar_{n+1}(B)}& \mbox{by Theorem \ref{theorem:EachTheoryIsApointInIgnatievsModel}.}\\
\end{array}
\]
Let $\pi$ be some $\Pi^0_{n+1}$ sentence.
We have thus seen that if $T_n^{Ar_{n+1}(B)}\vdash \pi$, then 
$T_{n+1}^A + T_n^B\vdash \pi$.

For the other direction, suppose that $T_{n+1}^A + T_n^B\vdash \pi$ for some $\pi \in \Pi^0_{n+1}$. We wish to see that $T_n^{Ar_{n+1}(B)}\vdash \pi$. We start with an application of Theorem \ref{theorem:EachTheoryIsApointInIgnatievsModel} and see:
\[
\begin{array}{llll}
T_{n+1}^A + T_n^{r_{n+1}(B)} &\equiv_{n+1} &T + A + {r_{n+1}(B)} & \\
	& \equiv_{n+1} & T + Ar_{n+1}(B) & \ \ \ \ \ \ \ \ \ \ \ \ \ \ \ \ \ \ \ \ \ \ \ \ \ \ \ (*)\\
	& \equiv_{n+1} & T_{n+1}^A + T_n^{Ar_{n+1}(B)}& \mbox{by Theorem \ref{theorem:OmegaSequenceInTuringProgressions}.}\\
\end{array}
\]
As $A \geq_{n+1}h_{n+1}(B)$, we see that $Ar_{n+1}(B)\geq_n h_{n+1}(B)r_{n+1}(B)$, whence $Ar_{n+1}(B)\geq_n  B$. Consequently, $T_n^{Ar_{n+1}(B)} \vdash T_n^B$. Thus, if $T_{n+1}^A + T_n^B \vdash \pi$, then also $T_{n+1}^A + T_n^{Ar_{n+1}(B)} \vdash \pi$ and by $(*)$, we see $T+ Ar_{n+1}(B) \vdash \pi$. Since $\pi \in \Pi^0_{n+1}$ we get by one more application of Theorem \ref{theorem:OmegaSequenceInTuringProgressions} that $T_n^{Ar_{n+1}(B)} \vdash \pi$ as was required.
\end{proof}
We note that the assumption $A\geq_{n+1} h_{n+1}B$ does not give us any information about the relations $\geq_{n+m}$ with $m>1$ at different coordinates in the Ignatiev sequence as these signs can switch arbitrarily. For example, let $A = 220222$ and $B=2122$. We have that 
\[
\begin{array}{lll}
A & >_0 & B\\
h_1(A) & <_1 & h_1(B)\\
h_2(A) & >_2 & h_2(B).\\
\end{array}
\]
The next lemma takes care of the case $A\leq_{n+1} h_{n+1}(B)$.
\begin{lemma}\label{theorem:ConservationIgnatievStep}
Let $T$ be some $\Pi_{n+1}$-axiomatized theory. Moreover, let $A\in \Worms_{n+1}$, $B\in \Worms_n$ and suppose $A\leq_{n+1} h_{n+1}(B)$. Then, verifiably in $T$, we have 
\[
T_{n+1}^A + T_n^B \equiv_n T_n^{B}.
\]
\end{lemma}

\begin{proof}
One direction is immediate so we reason in $T$ and assume that $T_{n+1}^A + T_n^B \vdash \pi$ for some $\pi\in\Pi^0_{n+1}$; we set out to prove that $T_n^B\vdash \pi$. However, by assumption $A\leq_{n+1}h_{n+1}(B)$ so that 
\begin{equation}\label{equation:bigTPImpliesSmallerTP}
T_{n+1}^{h_{n+1}(B)} \vdash T_{n+1}^A.
\end{equation}
Using this, we obtain
\[
\begin{array}{llll}
T+ B &\equiv& T + h_{n+1}(B) + r_{n+1}(B) & \\
 & \equiv_{n+1} & T_{n+1}^{h_{n+1}(B)} + T_n^{r_{n+1}(B)}& \mbox{by Theorem \ref{theorem:EachTheoryIsApointInIgnatievsModel}}\\
 & \equiv_{n+1} & T_{n+1}^{h_{n+1}(B)} + T_n^{B}& \mbox{by Theorem \ref{theorem:EachTheoryIsApointInIgnatievsModel}}\\
 & \vdash & T_{n+1}^A + T_n^B & \mbox{by \eqref{equation:bigTPImpliesSmallerTP}}.
\end{array}
\]
On the other hand, $T+B \equiv_n T_n^B$ so that if $T_{n+1}^A + T_n^B\vdash \pi$, then $T+B \vdash \pi$ whence also $T_n^B \vdash \pi$ quot erat demonstrandum.
\end{proof}

Suppose that $U\equiv_n V$. As mentioned before, in general we do not have that $U+T\equiv_nV+T$. However, we do have the following easy but useful lemma.

\begin{lemma}\label{theorem:ConservativeExtensions}
(In $\ea^+$) Suppose $U\equiv_nV$ and $T\subseteq \Sigma_{n+1}$, then also $U + T\equiv_n V+T$.
\end{lemma}

\begin{proof}
Immediate from the (formalized) deduction theorem.
\end{proof}

%

\begin{lemma}
Let $T$ be a $\Pi_1$ axiomatized theory and let $\vec A \in \mathcal{I}_\omega$. We have, verifiably in $T$, that  $\bigcup_{i=0}^{n}T_i^{A_i}\equiv_m \bigcup_{i=0}^{m}T_i^{A_i}$ for $m\leq n$.
\end{lemma}

\begin{proof}
By induction using lemmata \ref{theorem:ConservationIgnatievStep} and \ref{theorem:ConservativeExtensions}.
\end{proof}

\begin{lemma}\label{theorem:IgnatievSequencesStableUnderIgnatievOperator}
Let $\vec A \in \mathcal{I}_\omega$. If $\bigcup_{i=0}^{n}T_i^{A_i} \vdash T_n^B$, then $B\leq_n A_n$ given that $T$ is consistent.
\end{lemma}

\begin{proof}
Suppose otherwise, that is $A_n< B$. Then by compactness, for a single sentence $\pi$ of complexity at most $\Pi_{n}$ we have that $T_n^{A_n} + \pi \vdash T_n^B$. Since $B>A_n$ we certainly have $T_n^{A_n} + \pi \vdash \la n\ra_{T_n^{A_n}}\top$ whence, by provable $\Sigma_{n+1}$ completeness also $\la n\ra_{T_n^{A_n}}\pi$. Since the latter is equivalent to $\la n\ra_{T_n^{A_n} + \pi}\top$ we get by G\"odel's second incompleteness theorem for $n$-provability that $T_n^{A_n} + \pi$ is inconsistent.
\end{proof}

\begin{theorem}
Let $\vec A \in \mathcal{I}_\omega$. We have that ${\sf tt}((\vec A)_{\sf tt}) = \vec A$.
\end{theorem}

\begin{proof}
We need to see that $|(\vec A)_{\sf tt}|_m = A_m$. This follows directly from the previous lemmas.
\end{proof}

\subsection{Ignatiev's model: A roadmap to conservation results}

Now that we know that the $\Pi_n$ consequences of formal sub-theories between $\ea^+$ and $\pa$ correspond to points in Ignatiev's model, this gives us a nice way of collecting all we know about the $\Pi_n^0$ consequences of these fragments of arithmetic into a picture: We write the corresponding theories $T$ next to the nodes in the model that correspond to the spectra $tt(T)$ of $T$. If we have two theories $T_1$ and $T_2$, we can readily read off the amount of $\Pi_{n+1}$ conservation between them: look at the largest $n$ such that $\big(tt(T_1)\big)_n = \big(tt(T_2)\big)_n$.

For example, it is known that $\isig{1} \equiv \la 2\ra_{\ea^+} \top$ which would correspond to the point $\la \omega^\omega, \omega, 1\ra$. Similarly, \PRA will correspond to $\la \omega^\omega, \omega, 0\ra$ (see \cite[Corollary 4.14]{Beklemishev:2005:Survey}\footnote{There is a minor detail in that this results is formulated over $\ea$ and we work over $\ea^+$. Basic observations from \cite{Schmerl:1978:FineStructure} show that these differences are inessential for our limit ordinals.}). Parson's Theorem is readily read off from the picture (see Figure \ref{fig:ignatiev}) since $tt(\PRA)_1 = tt(\isig{1})_1 = \omega$ so that $\PRA \equiv_{\Pi^0_2} \isig{1}$.

\begin{figure}
\begin{center}
\input{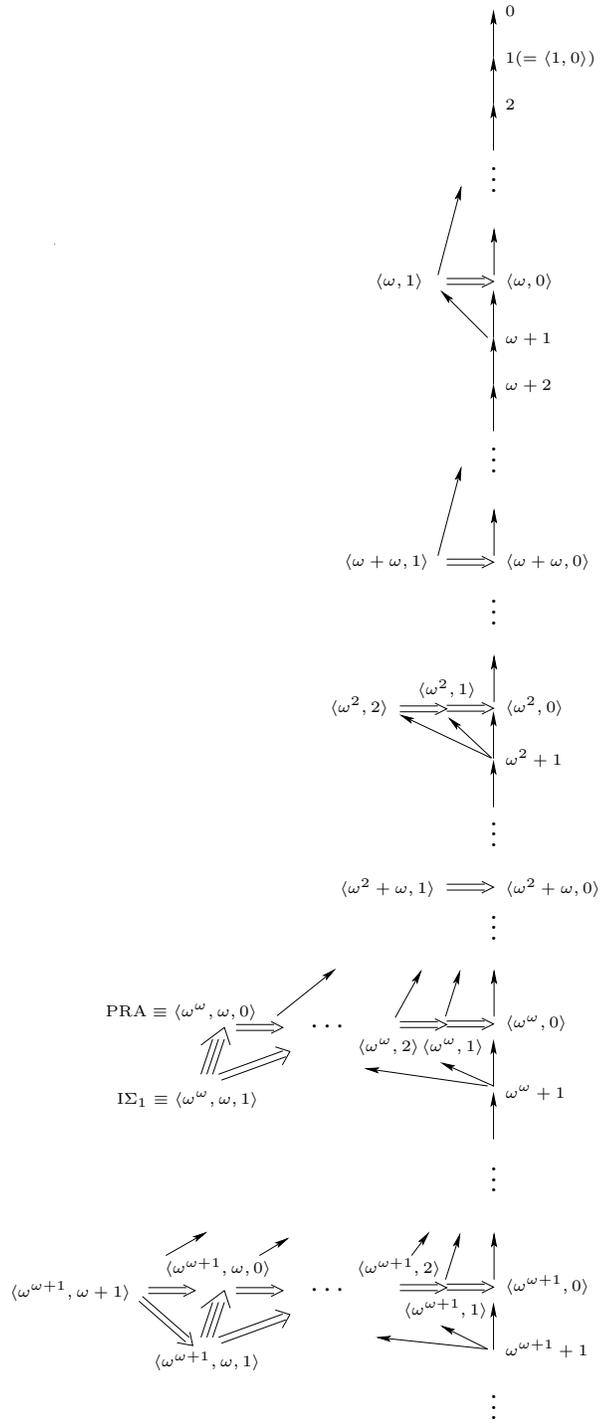}
\caption{Ignatiev's model: the $\succ_0$ relation is represented by a single arrow, $\succ_1$ by a double and $\succ_2$ by a triple arrow.} \label{fig:ignatiev}
\end{center}
\end{figure}

However, the picture does not tell us which theories proves what kind of consistency of which other theory. For example, although there is an arrow between $\isig{1}$ and \PRA, it is clear that \isig{1} proves no consistency of \PRA. In \cite{HermoJoosten:2015:TuringSchmerl} a model is presented where, apart from the conservation, one can directly see and compare the consistency strength of the Turing progressions depicted in that model.  We defer the further filling out of the Ignatiev model to a later paper.

\section*{Acknowledgement}
I am very grateful to Lev Beklemishev for teaching me much while being his student and for many fruitful discussions after that. For a long time Lev has advertised the analogy between Taylor expansions of well behaved functions on the one hand and characterizations of well-behaved formal theories in terms of reflection principles/consistency statements on the other hand. When I was still a PhD student I wondered what this could mean. The current paper is one further step towards my understanding of that question.

\bibliographystyle{plain}
\bibliography{References}

\end{document}